\documentclass[12pt]{article}
\usepackage{amsmath,amssymb,amsfonts}  
\usepackage{amsthm}
\usepackage{graphicx,psfrag,epsf}
\usepackage{enumerate}
\usepackage{natbib}
\usepackage{url} 
\usepackage{mathptmx}      
\usepackage{microtype}
\usepackage{xspace}
\newcommand{\blind}{0}

\usepackage{graphicx}
 \usepackage{mathptmx}      
\usepackage{amsmath,amssymb,amsfonts}         
\usepackage{url}

\usepackage{xspace}
\usepackage{microtype}

\usepackage{color}
\usepackage{dsfont}
\usepackage{bm}
\usepackage{algorithm}
\usepackage{algpseudocode}

\graphicspath{{./pictures/},{./}}

\DeclareMathOperator*{\argmin}{argmin}
\DeclareMathOperator*{\argmax}{argmax}

\newcommand*{\SBP}{SBP\xspace} 
\newcommand*{\dist}[1]{\ensuremath{\mathcal{P}(#1)}\xspace} 
\newcommand*{\distemp}{\ensuremath{\mathcal{P}}\xspace} 
\newcommand*{\prob}[1]{\ensuremath{\mathbb{P}(#1)}\xspace}

\newcommand{\ind}[1]{\raisebox{-0.3mm}{\scalebox{1.2}{\ensuremath{\mathds{1}}}}\ensuremath{\langle #1\rangle}}

\newcommand*{\sub}{\ensuremath{\subseteq}} 
\newcommand*{\bmid}{\ensuremath{\ \Big\vert\ }}

\newcommand*{\rel}[1]{\ensuremath{\mathfrak{r}(#1)}}

\newcommand*{\vertcov}{\ensuremath{\mathcal{D}}}

\newcommand*{\cardi}[1]{\ensuremath{\##1}\xspace}
\newcommand*{\primeprob}[1]{\ensuremath{\mathfrak{R}(#1)}}
\newcommand*{\SampleProb}[2]{\ensuremath{\mathfrak{S}(#1,#2)}}

\newcommand*{\supmarginal}[1]{\ensuremath{\mathfrak{B}(#1)}}
\newcommand*{\stepR}{\texttt{stepR}\xspace}
\newcommand*{\greedy}{\texttt{Greedy}\xspace}
\newcommand*{\limi}[1]{\ensuremath{\lim_{#1\rightarrow \infty}}\xspace} 
\newcommand*{\cp}{changepoint\xspace}
\newcommand*{\cps}{changepoints\xspace}

\makeatletter
\newcommand*{\defeq}{\mathrel{\rlap{\raisebox{0.3ex}{$\m@th\cdot$}}\raisebox{-0.3ex}{$\m@th\cdot$}}=}
\makeatother
{\bf}{\rm}
{\bf}{\rm}
\newtheorem{definition}{Definition}{\bf}{\rm}
\newtheorem{theorem}{Theorem}{\bf}{\rm}
\newtheorem{lemma}{Lemma}{\bf}{\rm}
\newtheorem*{heuristic}{Greedy}{\bf}{\rm}
\newtheorem*{ILP_problem}{ILP}{\bf}{\rm}
{\bf}{\rm}
\begin{document}

\def\spacingset#1{\renewcommand{\baselinestretch}%
{#1}\small\normalsize} \spacingset{1}


\if0\blind
{  
  \title{\bf Simultaneous Credible Regions for Multiple Changepoint Locations}
  \author{Tobias Siems\thanks{
    The authors would like to thank the following persons for their useful hints and corrections:
Lisa Koeppel, Paul Fearnhead, Nicolas Wieseke, Johann Jakob Preu\ss, Lawrence Bardwell, Areesh Mittal, Poppy Miller and Jamie-Leigh Chapman. The first author gratefully acknowledges support from the Landesgraduiertenf\"orderung, Greifswald and the DAAD (ID 57266578).
	}, Marc Hellmuth, Volkmar Liebscher \\ \smallskip
    Department of Mathematics and Computer Science\\
			University of Greifswald}
  \maketitle
} \fi

\if1\blind
{
  \bigskip
  \bigskip
  \bigskip 
  \begin{center}
    {\LARGE\bf Title}
\end{center}
  \medskip
} \fi

\bigskip
\begin{abstract}
Within a Bayesian retrospective framework, we present a way of examining the distribution of \cps through a novel set estimator.
For a given level, $\alpha$, we aim at smallest sets that cover all \cps with a probability of at least $1-\alpha$.
These so-called smallest simultaneous credible regions, computed for certain values of $\alpha$, provide parsimonious representations of the possible \cp locations.
In addition, combining them for a range of different $\alpha$'s enables very informative yet condensed visualisations. 
Therewith we allow for the evaluation of model choices and the analysis of \cp data to an unprecedented degree.
This approach exhibits superior sensitivity, specificity and interpretability in comparison with highest density regions, marginal inclusion probabilities and confidence intervals inferred by \stepR. 
Whilst their direct construction is usually intractable, asymptotically correct solutions can be derived from posterior samples. 
This leads to a novel NP-complete problem.
Through reformulations into an Integer Linear Program we show empirically that a fast greedy heuristic computes virtually exact solutions.
\end{abstract}

\noindent%
\vfill

\newpage
\spacingset{1.45} 
\sloppy

\section{Introduction}
\label{chap:intro}
Detecting \cps in time series is an important task.
For example, while observing the gating behavior of ion channels \citep{struct_change_ion_channel,modal_gating_ivo}.
There are many algorithms for, and scientific publications on, detecting multiple \cps in time series, 
such as frequentist approaches \citep{friedrich, frick_munk_sieling} and Bayesian approaches \citep{exact_fearnhead, ocpd, fearnhead_online}. 
A more exhaustive overview of existing methods can be found in \cite{eckley2011analysis}.

Detecting \cps in a time series usually comes down to deciding on a set of \cp locations.
Thus, Bayesian frameworks aim to infer a set valued random variable that gives a reasonable representation of this decision \citep{exact_fearnhead, ocpd, lai2011simple}.
The non-deterministic nature of these so-called \emph{posterior random \cps} expresses the uncertainty of their location.

\cite{Rigaill2012} illustrates this uncertainty by means of a Bayesian model with exactly two \cps.
They plot for all  possible pairs of time points  the posterior probability of being these \cps.
The results indicate both that posterior random \cps are highly dependent and that generally more than one combination is likely. 
Unfortunately, this approach is not suitable to monitor the distribution of more than two \cps.
It is a crucial fact that  the space of possible \cp locations is very high-dimensional even for time series of moderate size. 
Thus, an extensive exploration is a nontrivial task.

In Bayesian research, summaries of \cp locations, uncertainty measurements or model selection criteria are often provided by means of marginal inclusion probabilities.
\cite{Rigaill2012} gives a general consideration of this approach, but it has always enjoyed great popularity in the \cp community. 
See for example \cite{perreault2000retrospective, lavielle2001application, tourneret2003bayesian, exact_fearnhead, hannart2009bayesian, fearnhead_dependency, lai2011simple, aston2012implied, nam2012quantifying}.
Marginal inclusion probabilities, as shown in \cp histograms, simply consist of the probabilities for a \cp at each time point and thus, they are point wise statements.
However, due to the uncertainty of their location, (even single) \cps cannot be explored comprehensively by point wise statements.
On these grounds, in this paper we present a novel approach, that incorporates all possible \cp locations simultaneously.

Let $y=(y_1,\ldots,y_n)$ be a time series and let $C\sub\{1,...,n\}$ be the posterior random \cps.
A region that contains all \cps simultaneously with a probability of at least $1-\alpha$ is an $A\sub\{1,\ldots,n\}$ with $\prob{C\sub A}\geq 1-\alpha$.
We call such an $A$ a \emph{simultaneous $\alpha$ level credible region} or simply \emph{credible region} if there is no risk of confusion.
It provides a sensitive assessment of the whole set of possible \cp locations.
However, in order to obtain specific assessments as well, we seek a smallest such region.
This means we seek an element of
\begin{align}
\label{eq:minprob}
\primeprob{\alpha,C}\defeq\argmin\limits_{A\sub \{1,\dots,n\}}\Big\{\cardi{A}\bmid \prob{C\sub A}\geq 1-\alpha\Big\}
\end{align} 
where \cardi{A} is the cardinality of $A$.

\begin{figure}[ht]
\begin{center}
\includegraphics[width=\linewidth]{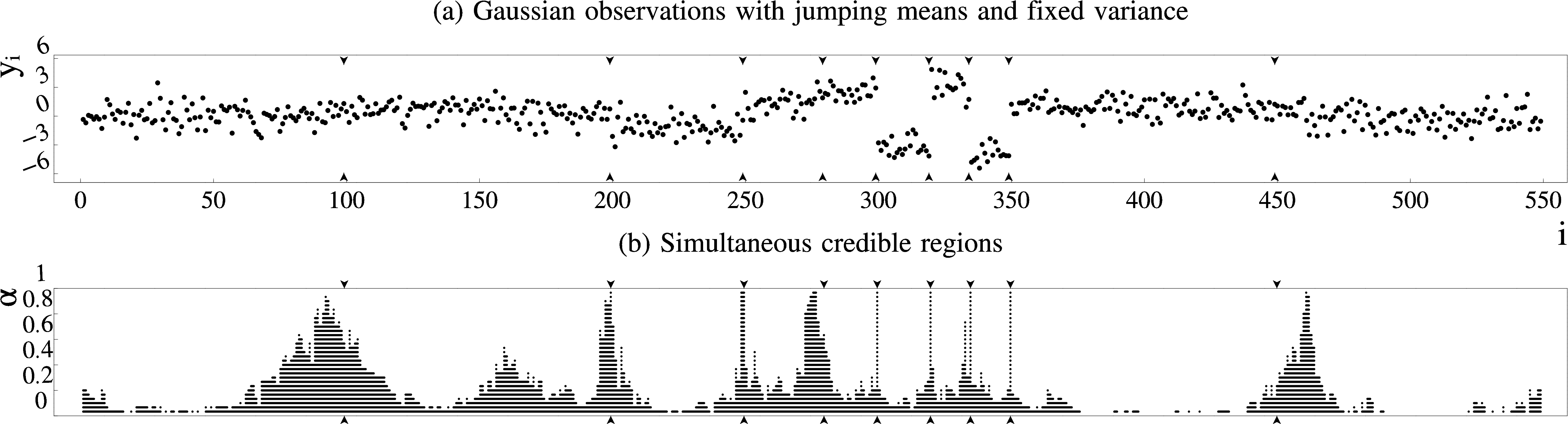} 
\caption{(a) A time series.
(b) 29 approximated smallest simultaneous $\alpha$ level credible regions, each of them   drawn as a broken horizontal line.}
\label{fig:introduction}
\end{center}
\end{figure}

Figure \ref{fig:introduction} demonstrates smallest credible regions by means of an example. 
The data points in (a) where drawn independently from a Gaussian distribution having a constant variance of 1 and mean values that are subject to successive changes.  
The true \cps are marked by small vertical arrows.
To build an exemplary Bayesian model here, we  choose a prior for the \cp locations and mean values:
The time from one \cp to the next is geometrically distributed with success probability $\frac{3}{550}$ and the mean values are distributed according to $\mathcal{N}(0,25)$.

Smallest credible regions are visualised in (b). For each $\alpha\in\Big\{\frac{1}{30}, \ldots, \frac{29}{30}\Big\}$ the plot shows one such (approximated) region as a collection of horizontal lines.
The value of $\alpha\in[0,1]$ controls the size of the regions and thus, governs the trade of between specificity and sensitivity.

The broadness of the credible regions around a true \cp expresses the uncertainty of its locations.
At the same time, a pointed shape reveals that the model is in favor of certain \cp locations.
Existing visualisation techniques, like \cp histograms, are unable to go beyond these two estimates.
However, it is also crucial to get an impression of the sensitivity of the model towards a true \cp.
We can examine this by looking at the height of the peak that relates to the true \cp.
The higher the peak, the higher the model's sensitivity.
We refer to this height as the importance of the true \cp.
The importances of the true \cps in Figure \ref{fig:introduction} are always larger than 0.9.

Of course, importance (as well as broadness and shape) can also be used for \cp data without true \cps.
There, we look at the peaks that belong to the features of interest.
Examples for features, that may occur in practice, are changes in mean, variance, slope or any other change in distribution. 
Furthermore, anomalies like outliers are of concern as well although they are usually supposed to be skipped by the model.
Figure \ref{fig:introduction} shows a nuisance feature in form of a small irregularity at around 170 with an importance of around 0.4.

By means of importance we can conveniently evaluate if the model is sensitive towards the desired features but skips the nuisance ones.
Conversely, we can also detect the relevant features in \cp data on the basis of a given model.
Most notably, this does not require any previous knowledge about the number of features or their positions.
This novel concept is one of the main outcomes of the present paper. 
It allows for a much more detailed analysis of \cp models and \cp data than hitherto possible.
Figure \ref{fig:introduction}(b) demonstrates this.
It shows that the \cp model in use is sensitive towards the desired true \cps and specific towards random distortions in the data.

The outline of this paper is as follows.
In Section \ref{chap:statapp} we consider the above problem from a general statistical and algorithmic viewpoint. 
We examine alternative approaches in Section \ref{chap:alt_approc}.
Section \ref{chap:simcredchanpoi} starts with an overview over sampling strategies in \cp models and deals with importance in a broader and more formal way.
Afterwards, we compare our results with the existing approaches.
We examine Dow Jones returns and demonstrate how credible regions can be used in order to perform model selection in Section \ref{chap:gaussian_change_in_variance}.
Finally, we discuss our results in Section \ref{chap:discussion}.

\section{The Sample Based Problem}
\label{chap:statapp} 

In this section we investigate the computational and mathematical foundations of our approach.
We show how to approximate credible regions in an asymptotic manner, examine the complexity and elaborate suitable algorithms.

We are given an arbitrary random set $C\sub\{1,...,n\}$.
Let $\distemp$ be the distribution of $C$, i.e. $\dist{A}=\prob{C=A}$ for all $A\sub\{1,...,n\}$. 
In this section, we are mainly concerned with the task of finding an element of \primeprob{\alpha, C} (see Equation (\ref{eq:minprob})), where $\alpha$ can take any value in $[0,1]$.

Deriving $\prob{C\sub A}$ is generally intractable since it arises from a summation of $2^{\#A}$ probabilities. 
Likewise, finding an element of \primeprob{\alpha,C} requires a search over the $2^n$ subsets of $\{1,\dots,n\}$. 
Hence, we address this problem in an approximate manner by using relative frequencies instead of probabilities.
\begin{definition}
For $A\sub\{1,\dots,n\}$ and $s_1,\dots,s_m\sub\{1,...,n\}$, let 
\begin{align*}
\rel{A,s_{1:m}}\defeq \frac{1}{m}\sum_{i=1}^m\ind{s_i\sub A}
\end{align*}
where $\ind{\dots}\in \{0,1\}$ is the indicator function which is equal to $1$ iff the bracketed condition is true.
$\rel{A,s_{1:m}}$ corresponds to the relative frequency with which the subsets are covered by $A$.
\end{definition}
In the following theorem, we will show that having independent samples $s_1,\dots,s_m\sub\{1,...,n\}$ from $\distemp$,
we can approximate an element of \primeprob{\alpha,C} by finding an element of
\begin{align*}
\SampleProb{\alpha}{s_{1:m}}\defeq\argmin\limits_{A\sub \{1,\dots,n\}}\Big\{\cardi{A}\bmid \rel{A,s_{1:m}}\geq 1-\alpha\Big\}
\end{align*}
We denote this problem as the \emph{Sample Based Problem (\SBP)}.

\newcommand*{\solutiontolem}{\mathcal{A}}
\begin{theorem}
\label{lem:convergence}
Let $S_1,S_2,\dots\sub\{1,...,n\}$ be independent random sets distributed according to
\distemp and $\alpha\in[0,1]$. 
If there exists an $\solutiontolem\in\primeprob{\alpha,C}$ with $\prob{C\sub
\solutiontolem}>1-\alpha$, then $\SampleProb{\alpha}{S_{1:m}}\sub\primeprob{\alpha,C}$ eventually almost surely.
\end{theorem}

\begin{proof}[Proof]
Let $A\sub\{1,\ldots,n\}$, since the $\ind{S_i\sub A}, i=1,2,\dots$ are
i.i.d. with finite expectations, the strong law of large numbers states that
$\limi{m}\rel{A,S_{1:m}}=\prob{C\sub A}$ almost surely. 
Since $2^{\{1,\ldots,n\}}$ is finite, $\rel{A,S_{1:m}}$ even converges to $\prob{C\sub A}$ for all $A\sub\{1,\ldots,n\}$ almost surely.

Let $s_1,s_2,\ldots\sub\{1,...,n\}$ be an arbitrary but fixed sequence with $\limi{m}\rel{A,s_{1:m}}=\prob{C\sub A}$ for all $A\sub\{1,\ldots,n\}$.
For $A\sub\{1,\dots,n\}$ with $\prob{C\sub A}< 1-\alpha$ we can pick an $m_A\in\mathbb{N}$ such that
$\rel{A,s_{1:m}}<1-\alpha$ for all $m>m_A$.
Hence,
$\SampleProb{\alpha}{s_{1:m}}\sub\{A\sub\{1,\dots,n\}\mid\prob{C\sub A}\geq 1-\alpha\}$
holds for all $m>m_1\defeq\max\{m_A\mid \prob{C\sub A}<1-\alpha\}$.
Furthermore, since $\prob{C\sub \solutiontolem}>1-\alpha$, we can pick an $m_0\in\mathbb{N}$ with
$\rel{\solutiontolem,s_{1:m}}>1-\alpha\ $ for all $m>m_0$. 
Thus, by choosing $\ell=\max\{m_0,m_1\}$, we obtain $\SampleProb{\alpha}{s_{1:m}}\sub\primeprob{\alpha,C}$ for all $m>\ell$.
\end{proof}

Since the exact value of $\alpha$ is usually irrelevant, the condition that there exists a certain solution can in most cases be neglected.

\subsection{Reformulation of the Sample Based Problem and its complexity}
\label{chap:impl_and_complexity}
In this section we consider the complexity of \SBP, formulate it as an \emph{Integer Linear Program} (ILP)
and introduce a fast and fairly accurate approximation by a greedy method. 
The next theorem shows that there is no hope
to find a polynomial time algorithm to solve \SBP unless $P= NP$ \citep{Garey_johnsons}.
\begin{theorem}
\label{theo:np}
The (decision version of) \SBP is NP-complete.
\end{theorem} 
The proof is provided in the supplementary material. 
To prove Theorem \ref{theo:np}, we show that there is a one-to-one relationship between our subsets of $\{1,...,n\}$ and hypergraphs. In contrast to (the better-known) graphs, 
hypergraphs can contain edges with more than two vertices \citep{voloshin,berge_hypergraphs} and thus, serve as a natural generalization of graphs. 
In particular, we show in the supplement that there is a close relationship between a particular 
optimization problem for hypergraphs and \SBP. 
Eventually, this allows us to prove the NP-hardness of \SBP and to highlight the close connection between hypergraphs and random subsets of finite sets.

\SBP is NP-hard and thus, there is no polynomial-time algorithm to 
optimally solve a given \SBP-instance. Nevertheless, \SBP 
can be formulated as an ILP and thus, \SBP becomes accessible to highly efficient 
ILP solvers \citep{ilp_bench}. Such ILP solvers can be employed 
to optimally solve at least moderately-sized \SBP-instances. 
The ILP formulation is as follows. 
\begin{ILP_problem}
\label{rem:simplesbp}
Following the notion of Section \ref{chap:statapp}, 
we are given a set of samples $s_1,\ldots,s_m\sub\{1,...,n\}$ and an $\alpha\in[0,1]$.
We declare binary variables $U_i,F_j\in\{0,1\}$ for all $i\in\{1,\ldots,n\}$ and $j\in\{1,\ldots,m\}$.
For $i\in\{1,\ldots,n\}$ let $\vertcov(s_{1:m},i)\defeq\{j\mid i\in s_j\}$ denote the set of samples that contain $i$.
Now define the following constraints
\begin{align*}
&\text{\textbf{\emph{(I)}}}\quad\sum\limits_{j=1}^m F_j\geq m\cdot(1-\alpha)\\
&\text{\textbf{\emph{(II)}}}\quad\forall i\in\{1,\ldots,n\}:\ \sum\limits_{j\in \vertcov(s_{1:m},i)} (1-F_j)\geq  \cardi{\vertcov(s_{1:m},i)}\cdot (1-U_i)
\end{align*}
under which the objective function $\sum_{i=1}^n U_i$ needs to be minimized. 

Having computed an optimum, $A=\{i\mid U_i=1\}\in\SampleProb{\alpha}{s_{1:m}}$, i.e. this set is  a solution of \SBP.
\end{ILP_problem}

The binary variable $F_j$ represents the $j$-th sample. 
If  $F_j = 1$ then $s_j\sub A$.
Thus, the constraint \textbf{(I)} states that $A$ covers at least $m\cdot(1-\alpha)$ samples, which is equivalent to $\rel{A,s_{1:m}}\geq 1-\alpha$ in the definition of $\SampleProb{\alpha}{s_{1:m}}$.
The constraint \textbf{(II)} states that time point $i$ can only be dropped from $A$, if $i\in s_j $ implies that $F_j=0$.
From the perspective of \SBP these constraints are obviously necessary and the proof of sufficiency is provided in the supplement.

A benchmark of several ILP solvers can be found in \cite{ilp_bench}.
Following this advice, we use CPLEX V12.6.3 for Linux x86-64 \cite{cplex} on a Lenovo Yoga 2 Pro (8GB Ram, 4 x 1.8GHZ Intel CPU) to solve our ILP instances.

\subsection{A greedy heuristic}
\label{chap:greedy}
To provide an alternative way to address \SBP, 
we now resort to a simple greedy heuristic.
This approach starts with the whole set of time points and 
greedily removes all time points iteratively. 
The greedy rule used here chooses a time point that is contained in the smallest number of samples.
In the subsequent steps, these samples will be ignored.

\begin{heuristic}
Let $A_0\defeq \{1,\ldots,n\}$. Compute $A_{\ell+1}=A_\ell\backslash\{k_{\ell+1}\}$ iteratively where
$k_{\ell+1}\in\argmax_{i\in A_\ell}\big\{\rel{A_\ell\backslash\{i\},s_{1:m}}\big\}$ until $A_{\ell+1}=\emptyset$.
Use $A_\ell$ as a solution proposal for \SBP for any $\alpha$ with $\rel{A_{\ell+1},s_{1:m}}<1-\alpha\leq \rel{A_\ell,s_{1:m}}$.
\end{heuristic}

\begin{algorithm}[hbt]
\begin{algorithmic}[1]
\footnotesize 
\State{\textbf{Input:} List $s[j]$ that represents the sample $j$, $1\leq j\leq m$}
\State{\hspace{1cm} Sorted lists $\mathcal{C}[i]$ that represents the samples that contain time point $i$, $1\leq i\leq n$}
\State{\textbf{Output:}} $A_1,\ldots,A_n$ \Comment{ $A_{\ell+1}=A_\ell\backslash\{k_{\ell+1}\}$ with $k_{\ell+1}\in\argmax_{i\in A_\ell}\big\{\rel{A_\ell\backslash\{i\},s_{1:m}}\big\}$ }
\State{$A_0=\{1,\ldots,n\}$}
\For{$\ell=0,\ldots,n-1$}\Comment{Discard all time points iteratively}
	\State{$i=\argmin\big\{\mathcal{C}[k]$.length$\ \mid\ k\in A_\ell\big\}$} \label{alg:ass1}\Comment{Complexity $\mathcal{O}(n)$}
	\State{$A_{\ell+1}=A_\ell\backslash\{i\}$}
	\For {$j=1,\ldots,\mathcal{C}[i]$.length and $k=1,\ldots,s[j]$.length}  \Comment{Remove all samples that contain $i$ from $\mathcal{C}$}
		\State{$\quad\mathcal{C}\Big[s[j][k]\Big]$.remove($j$)}\label{alg:rem1}\Comment{Removing has a complexity of $\mathcal{O}(\log(m))$} 
	\EndFor
\EndFor
\end{algorithmic}
\caption{\greedy-Algorithm for \SBP}\label{alg:greedy}
\end{algorithm}
 
Algorithm \ref{alg:greedy} shows pseudocode of \greedy.
The runtime of this algorithm is  $\mathcal{O}(n^2m\log(m))$.
To see this, observe that the first for-loop runs $n$ times
and the assignment in Line \ref{alg:ass1} needs $\mathcal{O}(n)$ time. 
The second for-loop runs $\mathcal{O}(mn)$ times and the removal of
an element in Line \ref{alg:rem1} can be done in  $\mathcal{O}(\log(m))$ time. 
Thus, the  overall time-complexity is 
$\mathcal{O}(n^2 + n^2m\log(m)) = \mathcal{O}(n^2m\log(m))$.

Although, the worst case runtime-complexity is more or less cubic, Algorithm \ref{alg:greedy} can be well applied in practice for the following reasons.
Assume that every sample has exactly $k$ elements uniformly distributed over $\{1,...,n\}$.
In this case, the number of samples containing $i$ is binomially distributed with parameters $\frac{k}{n}$ and $m$.
Thus, their average number is $m\cdot \frac{k}{n}$.
Finally, the average complexity reduces to $\mathcal{O}(k^2m)$ if we use a hash table to store elements of the $\mathcal{C}[i]$'s, a sorted list to store their lengths and assume that $n\log(n)<k^2m$.
In general, where the elements of the samples are non-uniformly distributed and $k$ only represents their maximal cardinality the runtime will be shortened even further.

\greedy provides credible regions for all $\alpha$ values at once. 
However, in contrast to solutions to the ILP, the resulting credible regions will be nested w.r.t. increasing $\alpha$ values.

\section{Alternative approaches}
\label{chap:alt_approc}
In this section we explain two alternative Bayesian approaches which can be derived from $C$ as well.
In the first approach we join highest density regions in order to infer credible regions. 
Highest density regions can be considered as a straightforward simultaneous tool to explore distributions of interest \citep{Held_sim_post}.
The second approach, marginal inclusion probabilities, is a straightforward point wise tool in the context of \cps.

Besides this, it should be noted that \cite{guedon2015segmentation} addresses uncertainty of \cp locations through the entropy of $C$.

\subsection{Highest density regions}
\label{chap:hdr}
A highest density region (HDR) is a certain subset of a probability space with elements having a higher density value than elements outside of it. 
Such a subset can be utilized to characterize and visualise the support of the corresponding probability distribution \citep{hyndman_computing}. 
In a Bayesian context, $(1-\alpha)$-HDR's are often used as simultaneous $\alpha$ level credible regions \citep{Held_sim_post}. 
In this section, we examine HDR's in general and for the case of random subsets of $\{1,...,n\}$.

Let $X$ be a random variable with a density $p$. 
For $\alpha\in[0,1]$, let $q_\alpha$ be an $\alpha$-quantile of $p(X)$, i.e.
$\prob{p(X)\leq q_\alpha}\geq \alpha$ and $\prob{p(X)\geq q_\alpha}\geq
1-\alpha$. 
\begin{definition} The set
$\{x\mid p(x)\geq q_\alpha\}$ is referred to as the \emph{$(1-\alpha)$-HDR} of $X$.
\end{definition}

The $(1-\alpha)$-HDR is a smallest subset of the state space with a probability of at least $1-\alpha$ \citep{box73}.

Now we consider $X=C$.
Our credible regions are subsets of $\{1,\dots,n\}$, whereas the $(1-\alpha)$-HDR of $C$ would be a subset of $2^{\{1,...,n\}}$. 
For the purpose of comparison, we join all successes in the HDR:
\begin{definition} 
Let $s_1,\dots,s_\ell\sub \{1,...,n\}$ be the
$(1-\alpha)$-HDR of $C$. We refer to $\bigcup_{i=1}^\ell s_i$ as
the \emph{joined $(1-\alpha)$-HDR} of $C$. 
\end{definition}
Let $A$ be the joined $(1-\alpha)$-HDR and $\distemp$ the distribution of $C$. Then
\begin{align*}
\prob{C\sub A}=\distemp\big(2^{A}\big)\geq \sum_{i=1}^\ell\dist{s_i}\geq 1-\alpha
\end{align*}
Even though, $s_1,\dots,s_\ell$ have high probabilities, $2^{\cardi{A}}$ may be much larger than $\ell$. 
Likewise, $\prob{C\sub A}$ might be considerably larger than $1-\alpha$. 
Thus, $A$ might be substantially larger than elements of $\primeprob{\alpha,C}$. 
In Section \ref{chap:simcredchanpoi} we will see that this happens, especially for small
$\alpha$.

Unfortunately, in many cases we cannot compute HDR's directly and therefore, we use an approximation scheme \citep{Held_sim_post}.
Let now $s_1,\dots,s_m\sub\{1,...,n\}$ be independent samples from the distribution of $C$ which are arranged in descending order according to their values under $\distemp$, i.e. $i>j\Rightarrow \distemp(s_i)\leq \distemp(s_j)$.
We use $\bigcup_{i=1}^\ell s_i$ with $\ell\defeq {\lceil m\cdot(1-\alpha)\rceil}$ as an approximate joined $(1-\alpha)$-HDR.

\subsection{Marginal inclusion probabilities}
\label{chap:pointwise}
Here we consider probabilities of the form $\prob{i\in C}$ for $i=1,\ldots,n$. 
They can be used to derive subsets of $\{1,\ldots,n\}$ that are closely related to credible regions.

\begin{lemma}
$\{i\mid \prob{i\in C}> \alpha\}$ is a subset of all elements of $\primeprob{\alpha,C}$.
\end{lemma}
\begin{proof}[Proof]
Since $\prob{i\in C}> \alpha$ implies $\prob{C\sub \{1,...,n\}\setminus i}< 1-\alpha$, $i$ has to be part of any $\alpha$ level credible region.
\end{proof}
The Bonferroni correction \citep{dunnett55} can be applied to construct a credible region. 
Let
$\supmarginal{\alpha, C}\defeq\Big\{i\bmid \prob{i\in C}>\frac{\alpha}{n}\Big\}$.
\begin{lemma}
$\prob{C\sub\supmarginal{\alpha, C})}\geq 1-\alpha$
\end{lemma}
\begin{proof}[Proof]
Let $A\defeq\supmarginal{\alpha, C}^c$. 
We conclude that $1-\prob{C\sub A^c}=\prob{C\cap A\not=\emptyset }\leq \sum_{i\in A} \prob{i\in C}\leq \alpha$. 
Thus, $\prob{C\sub A^c=\emptyset}\geq 1-\alpha$ applies.
\end{proof}

Point wise statements suffer from their inability to reflect dependencies. 
To see this, we assume that $\prob{\#C=1}=1$ and  $\prob{i\in C}=\frac{1}{n}$ for all $i=1,\ldots,n$. 
Single time points are strongly dependent, e.g. $i\in C$ implies $j\not\in C$ for all $j\neq i$.
Clearly, $\supmarginal{\alpha, C}=\{1,\ldots,n\}$ for all $\alpha<1$.
Moreover, $\{i\mid \prob{i\in C}> \alpha\}=\emptyset$ for $\alpha\geq \frac{1}{n}$.
In contrast to this, elements of $\primeprob{\alpha, C}$ become smaller if $\alpha$ becomes larger. More precisely, $\#A=\lceil (1-\alpha) n\rceil$ for $A\in\primeprob{\alpha,C}$. 
This is due to the fact that point wise statements ignore dependencies, whereas simultaneous statements incorporate them. 
Hence, in practice $\supmarginal{\alpha, C}$ may be much broader than elements of $\primeprob{\alpha, C}$ and $\{i\mid \prob{i\in C}> \alpha\}$ may be much smaller. 
In Section \ref{chap:simcredchanpoi} we will give an example that underpins this claim.

\section{Multiple changepoints and credible regions in more detail}
\label{chap:simcredchanpoi}

In this section we discuss sampling strategies in \cp models and we consider the concept of importance in a more formal way.
Afterwards we reconsider the aforementioned example in order to compare existing approaches with the credible region approach.
Finally, we investigate the performance of \greedy and the convergence speed of \SBP empirically.

\subsection{Existing sampling strategies}

Generating exact posterior samples of Bayesian multiple \cp models is a well studied task. 
\cite{exact_fearnhead} describes how to infer samples from \cp models that meet two conditions. 
Firstly, a \cp must induce independence between the random variables before and after the \cp.
This applies to all the considered examples in this paper. 
An important counterexample where the independence assumption does not hold is a piecewise linear regression with attached lines.
Secondly, a conjugacy assumption must be met. 
Generating one sample has a complexity of $\mathcal{O}(n^2)$, which can be reduced by pruning.
This sampling method was used in our \cp models, except for the Well-Log example in the supplement.

The algorithm described in \cite{exact_fearnhead} is also capable of sampling when \cps induce independence, but the prior jump distribution is not conjugated and thus, 
closed form representations of the posterior distribution might not be available.
This is done by using numerical integration, which may further impair the complexity.
By resorting to this method, in the Well-Log example we sample with a runtime complexity of $\mathcal{O}(n^3)$.

\cite{fearnhead_dependency} explain how to sample from \cp models that meet the conjugacy assumption but do not induce independence across \cps.
This is done by accepting inaccuracy arising through a pruning approach.
Finally, if the distribution does not meet the independence nor the conjugacy assumption, one may pursue an MCMC or SMC approach \citep{green1995, heard2017}.

\subsection{Importance revisited}
\label{chap:importance}

Characterizing the features in the data which should or should not trigger a \cp enables the formulation of \cp problems in the first place.
An intrinsic property of virtually every \cp problem is that the location of the \cp, placed as a consequence of a feature, is uncertain.
Hence, a feature in the data is usually related to a set of locations instead of a single location.
In order to deal with this fact, we require a simultaneous approach that considers all \cp locations that belong to the same feature in a coherent fashion.

To this end, as intimated in Section \ref{chap:intro}, importance provides an estimate of the sensitivity of \cp models towards features in \cp data.
However, it remains to substantiate this claim mathematically.
We start with our definition of sensitivity.
The sensitivity of a \cp model towards a feature in the data is defined as the probability with which posterior random \cps occur as a consequence of this feature.
Thus, if $A\sub\{1,...,n\}$ represents the set of time points which are related to the feature, the sensitivity equates to $\prob{C\not\sub A^c}$.
In general, determining the set $A$ is a fuzzy task. 
However, since \cp data usually exhibits the relevant features only occasionally, in most cases we can easily specify these sets to a sufficient extent.

In the introduction the importance of a feature is roughly defined as the height of the peak that relates to this feature.
This means we derive importance from a given range of credible regions.
Assume that we are given a whole range of credible regions $R_\alpha\in\primeprob{\alpha,C}$ for all $\alpha\in[0,1]$.
We now define the importance of a feature as $\hat\alpha=\inf\{\alpha\mid R_\alpha\sub A^c\}$, where $A$ is again the set of locations that belongs to the feature.
With this the following lemma applies.
\begin{lemma}
\label{lem:exclude}
The importance of a feature is an upper bound for the sensitivity of the \cp model towards this feature.
\end{lemma}
\begin{proof}
If there exists an $\epsilon>0$ with $R_{\alpha}\sub A^c$ for all $\alpha\in]\hat\alpha, \hat\alpha+\epsilon[$, we conclude
\begin{align*}
&1-\alpha\leq\prob{C\sub R_{\alpha}}\leq \prob{C\sub A^c}\text{ for all } \alpha\in]\hat\alpha, \hat\alpha+\epsilon[\\
&\Leftrightarrow\quad 1-\hat\alpha\leq \prob{C\sub A^c}
\quad\Leftrightarrow\quad\hat\alpha\geq\prob{C\not\sub A^c}
\end{align*}
In the case where $R_{\hat\alpha}\sub A^c$, we may apply the second part of the above equation directly.
\end{proof}

Since the credible regions are not necessarily nested, in some rare cases, the importance read by means of the above definition may not correspond to a global peak in the credible regions.
Furthermore, if credible regions are only drawn for an incomplete range of $\alpha$'s, we may slightly underestimate the heights of the peaks.
In exchange for these negligible inaccuracies, we obtain an intuitive, visual estimate that provides great insights into the distribution of changepoints and is thus of high practical value.

\begin{figure}[ht]
\includegraphics[width=\linewidth]{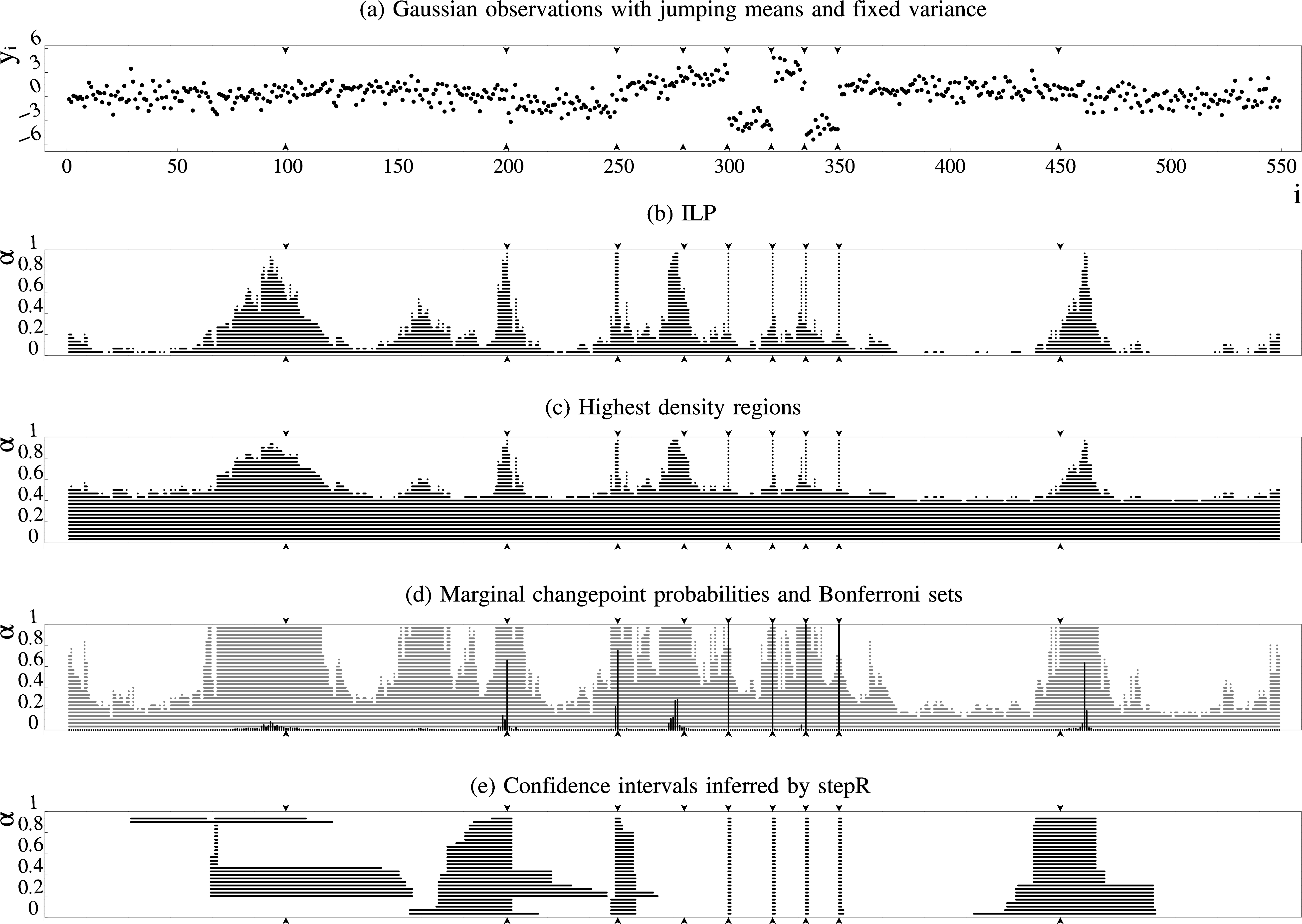}
\caption{(a) Simulated independent Gaussian observations with variance 1 and successive changes in mean. 
The true \cps are marked by vertical arrows.
(b) Solutions derived by solving the ILP using $10^5$ samples.
(c) Joined $(1-\alpha)$-HDR's using $10^7$ samples. 
(d) Marginal inclusion probabilities in black and $\supmarginal{\alpha, C}$ in gray.
(e) Joined confidence intervals inferred by \stepR.
}
\label{fig:data}
\end{figure}

\subsection{Exemplary comparison}
\label{chap:gaus_change_in_mean}

Figure \ref{fig:data}(a) displays the same dataset as Figure \ref{fig:introduction}(a).
There are 6 obvious true \cps at 200, 250, 300, 320, 335 and 350 having a large jump height and two true \cps at 280 and 450 with smaller jump heights of 1 and 0.8, respectively. 
Besides, there is a true \cp at 100 with an even  smaller jump from 0 to 0.5.

We now want to compare the credible regions approach with the highest density regions approach, marginal jump probabilities and confidence intervals inferred by the R package \stepR.
We use the same model as in the introduction.
Figure \ref{fig:data}(b) displays the credible regions with respect to $\alpha=\frac{1}{30},\ldots, \frac{29}{30}$ inferred by solving the ILP.

In the same fashion as before, Figure \ref{fig:data}(c) displays several approximated joined HDR's derived from $10^7$ samples. 
This confirms that joining the elements of an HDR leads to larger subsets of $\{1,\dots,n\}$ compared to elements of \primeprob{\alpha,C}. 
At an $\alpha$ level less than 0.4, the joined HDR already covers most of the time points and thus, holds no special information about the \cp locations anymore.

Figure \ref{fig:data}(d) displays the marginal inclusion probabilities, i.e. $\prob{i\in C}$ highlighted in black. 
Unfortunately, they don't reflect the possible set of \cps very well since they appear to be too sparse. 
Furthermore, due to a dispersal of the probabilities they are not able to express the sensitivity towards the true \cps (see for example the \cps at 100 and 280).
This leads to the conclusion that this approach suffers from a lack of interpretability and is less sensitive than the credible regions approach.

Additionally, Figure \ref{fig:data}(d) shows several credible regions corresponding to $\supmarginal{\alpha, C}$.
As prognosticated in Section \ref{chap:pointwise} these regions are very broad compared to smallest credible regions and thus, less specific.

Figure \ref{fig:data}(e) displays several joined confidence intervals inferred by \stepR \citep{stepR}. 
\stepR first estimates the number of \cps and produces one confidence interval for each \cp.
The plot shows the union of these confidence intervals.
Although confidence sets and credible regions are different by definition, they intend to make similar statements. 

Unfortunately, \stepR does not forecast a confidence interval for the true \cp at 280.
Furthermore, the disappearance of certain \cp locations at decreasing $\alpha$ values seems somewhat confusing and hard to make sense of.
So, the shapes, broadnesses and importances read from the confidence sets are if at all only of little informative value.
Therefore, the authors of this paper consider the confidence sets inferred by \stepR as unsuitable for this purpose, because they suffer from a grave lack of interpretability.

The supplement contains a collection of pictures similar to Figure \ref{fig:data}.
There, the data was repeatedly generated using the same \cp locations and mean values.
It becomes apparent, that these kind of plots, produced with \stepR, are almost consistently of poor quality.

In Section \ref{chap:importance} we show that the importance of a feature in the data is an upper bound for the sensitivity of the model towards this feature. 
We can examine the differences between these two estimates empirically by means of the above example.
Sensitivity and importance match for all the true \cps except the first one.
The importance of the first \cp is approximately 0.94, whereas its sensitivity with regard to the interval from 0 to 130 is approximately 0.72.
That's a deviation of around 0.2.
The small irregularity at around 170 results in an even bigger deviation of around 0.3.

\begin{figure}[ht]
\begin{center}
\includegraphics[width=\linewidth]{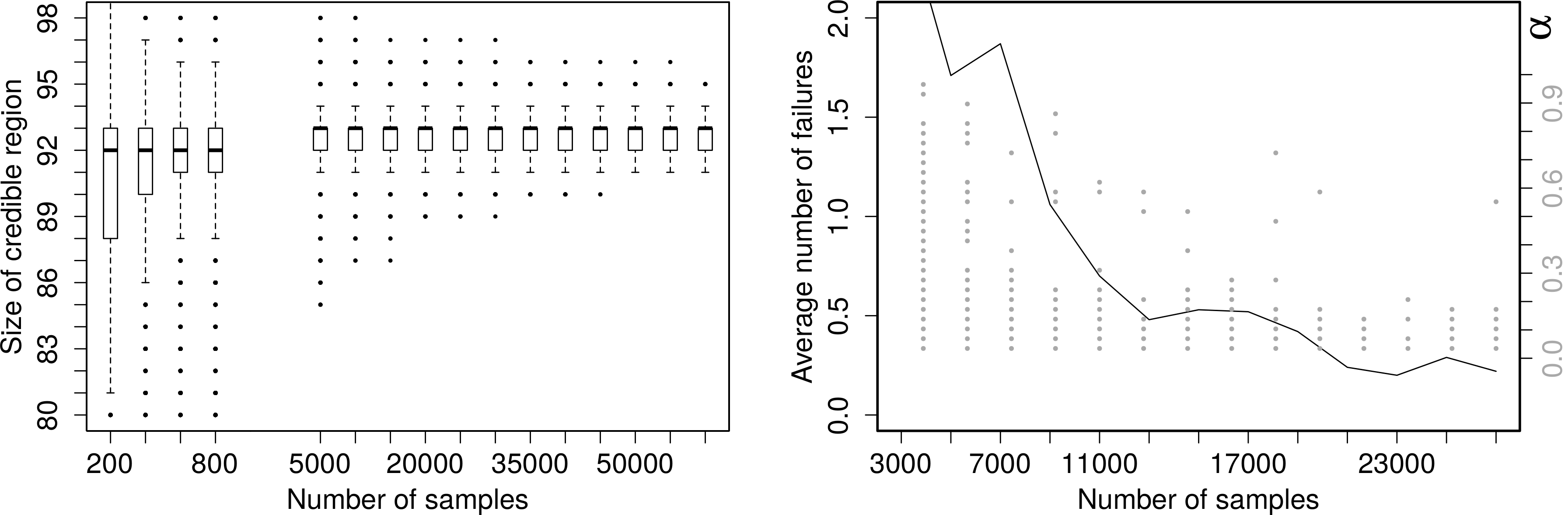} 
\caption{(left) An empirical convergence proof of \SBP. Boxplots over the sizes of the elements of $\SampleProb{0.3}{s_{1:m}}$ where $m$ is varying according to the x-axis.
(right) Illustration of the accuracy of solutions inferred by \greedy. 
The black graph illustrates the average number of failures within 29 different credible regions (left y-axis). 
The gray dots mark the $\alpha$ values (right y-axis) where failures occurred.}
\label{fig:boxplots}
\end{center}
\end{figure}

\subsection{Empirical proof of convergence and the accuracy of \greedy}
\label{chap:conv_and_accuracy}

Figure \ref{fig:boxplots} (left) demonstrates how solutions to \SBP evolve with increasing sample size.
For several sample sizes $m$ it shows boxplots (using 1000 repetitions) over the size of elements of $\SampleProb{0.3}{s_{1:m}}$.
It can be observed that the sizes of the credible regions increase with an increasing sample size. 
This is due to the fact that we need a certain number of samples to cover the possible \cp locations satisfactorily.
However, at the same time, we observe an increasing preciseness which is a result of Theorem \ref{lem:convergence}.

For different sample sizes we compared for each $\alpha\in\Big\{\frac{1}{30}, \ldots, \frac{29}{30}\Big\}$ the solutions provided by solving the ILP with those provided by \greedy.
In Figure \ref{fig:boxplots} (right), the black graph illustrates the average number of $\alpha$'s (using 100 repetitions) where \greedy failed to provide an optimal solution.
At higher sample counts, there are less than 0.3 of the 29 credible regions wrong. 
The gray points represent the $\alpha$ values (right axis) where the ILP was able to compute smaller regions than \greedy.

Hence, \greedy performs virtually exact on this \cp problem.
However, at smaller sample counts it gets more frequently outwitted by random.
Fortunately, this shows that if \greedy does not compute an ideal region for a certain $\alpha$, it will compute correct ones later again.
This is due to the fact that credible regions for \cp locations are roughly nested (like the solutions provided by \greedy). 
Only in some cases it is possible through solving the ILP to remove certain time points a little bit earlier than \greedy. 
Furthermore, because \greedy removes samples having a high number of \cps fast, it works well together with \cp models, since they prefer to explain the data through a small number of \cps.
Thus, we can conjecture that \greedy will perform just as well in most \cp scenarios.

\section{An example of use for model selection}
\label{chap:gaussian_change_in_variance}
\begin{figure}[ht]
\includegraphics[width=\linewidth]{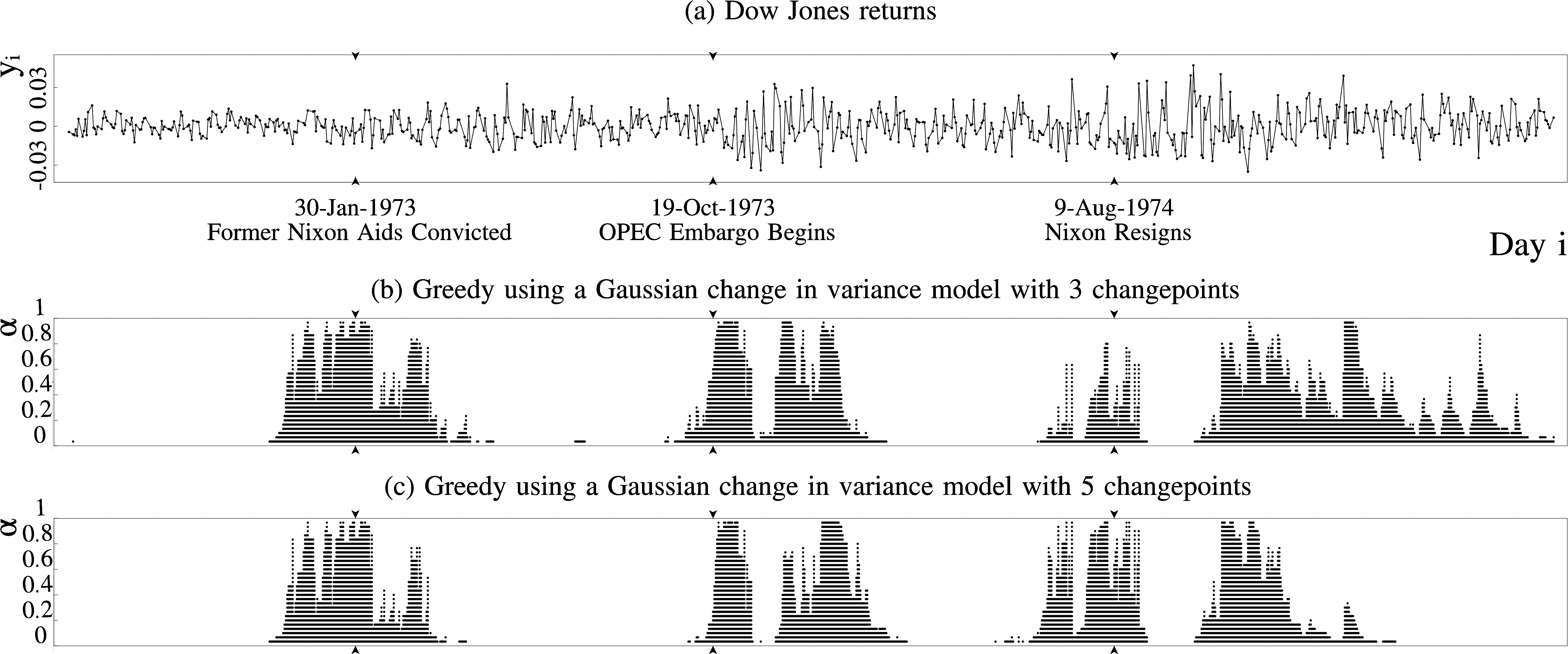}
\caption{(a) Dow Jones returns. (b) and (c) Different regions provided by \greedy.
}
\label{fig:dow_jones}
\end{figure}

Now we examine Dow Jones returns observed between 1972 and 1975 \citep{ocpd}, see Figure \ref{fig:dow_jones}(a). 
There are three documented events highlighted on January 1973, October 1973 and August 1974. 
The data is modeled as normally distributed with constant mean equal to 0 and jumping variances.
The variances are distributed according to an inverse gamma distribution with parameters 1 and $10^{-4}$ (see also \cite{ocpd}).
Instead of allowing a random number of \cps, this time we predetermine the number of \cps to 3 respectively 5 and our aim is to compare these two model choices by means of credible regions.

As we can see in Figure \ref{fig:dow_jones}(b), the regions become fairly broad especially in the last third of the picture, which gives rise to additional, nonsensical \cp locations.
The reason for this is that assuming only one \cp after the second event, yields to a misjudgment of the third or fourth variance. 
Thus, the third \cp becomes superfluous and its location highly uncertain.

In contrast, in the case of five \cps in (c) the illustration turns out to be much more differentiated.
The regions stay fairly narrow even for very small values of $\alpha$.
Just a quick look at the credible regions immediately reveals that fixing the number of \cps may impair the inference dramatically. 
Thus, whenever possible, we should use a model that allows the number of \cps to be determined at random.

\section{Discussion}
\label{chap:discussion}

In this paper we develop a novel set estimator in the context of Bayesian \cp analysis. 
It enables a new visualisation technique that provides very detailed insights into the distribution of \cps.
The resulting plots can be analysed manually just by considering the concepts of
broadness to assess the uncertainty of the \cp locations in regards to a feature,
shape to explore the \cp locations the model is in favor of to explain a feature, and
importance to get an idea about the sensitivity of the model towards a feature.
This greatly facilitates the evaluation and the adjustment of \cp models on the basis of a given dataset.
However, by means of these three concepts, we are also able to conveniently analyse \cp datasets on the basis of a predetermined \cp model.

We say that a credible region points to a feature in the data if the region contains at least one of the \cp locations that are triggered by the model as a consequence of this feature.
An $\alpha$ level credible region will point to all features having a sensitivity larger than $\alpha$.

Against this backdrop, we may derive a single $\alpha$ level credible region from a given model in order to analyse the features in a dataset. 
Thus, we need to choose $\alpha$ in such a way that the corresponding credible region points to the features of interest, but skips the uninteresting ones.
A good \cp model will represent the degree of interest through its sensitivity.
On that basis, we can choose $\alpha$ such that the requirements regarding accuracy and parsimony are met.
\cite{frick_munk_sieling} recommend to choose $\alpha=0.4$ for their method.
This is a reasonable choice.
However, it becomes apparent in Section \ref{chap:gaus_change_in_mean} that the difference between sensitivity and importance can be around 0.3 in difficult cases.
Thus, we recommend to choose $\alpha=0.7$ instead.

An essential quality of the \cp model is its willingness to jump since it highly affects the model's  sensitivity towards features.
In the supplement you can find several video files illustrating this through the success probability for the distribution of the sojourn time between successive \cps.
In this context, if \cp samples are available but the necessary modifications of the model are not feasible anymore,
a sloppy but perhaps effective way to obtain a representative credible region is to increase or decrease $\alpha$ until a reasonable coarseness comes about.
However, systematic defects as in the model depicted in Figure \ref{fig:dow_jones}(b) cannot be tackled in this way.

Credible regions provide valuable knowledge about groups of \cp locations that explain single features jointly.
Besides, they reveal a little bit about combinations of \cps with respect to more than one feature.
The shapes of the credible regions shown in Figure \ref{fig:data}(b) and \ref{fig:dow_jones}(c) suggest that the data can be explained very well through combinations of 9 respectively 5 \cps, which are limited to very few locations.
On the other hand, the medium importance of the feature around 170 in Figure \ref{fig:data}(b) shows that there is a notable alternate representation, which could be quite different to that of the true \cps.

Our theory so far is built on random sets which represent the posterior random \cps.
However, since we can specify a bijective function between subsets of $\{1,...,n\}$ and binary sequences in $\{0,1\}^n$, all the theory in this paper applies equally to sequences of binary random variables of equal length.
Thus, we may apply credible regions to ``Spike and Slab'' regression \citep{spike_slab_mitchel, george_mccullogh_gibbs, ishwaran2005} as well.
There, each covariate is assigned to a binary variable which determines if the covariate is relevant for explaining the responses.
However, covariates can generally not be described in terms of a time series and this could hinder a graphical analysis through credible regions.

Even though the construction of credible regions evoke acceptance regions from statistical testing, in this paper we do not intent to create a method in this direction.
As we can see in the example in Section \ref{chap:gaus_change_in_mean}, credible regions can be fairly broad for say $\alpha=0.05$.
While this would interfere with statistical testing purposes, it provides valuable insights into \cp models.

The ILP can be improved by introducing a constraint for each single \cp in the samples, instead of having constraints for larger sets of samples that share the same \cp location (Constraint II).
This is because for an ILP solver, many small constraints are easier to handle than a big one, that involves many variables.
However, due to high runtimes we do not recommend using the ILP.
While we are able to compute solutions to the Gaussian change in mean example, it is not possible to compute all 29 solutions to the Dow Jones example within a week.

To conclude, the authors of this paper highly advice the use of \greedy's credible regions to evaluate and justify \cp models. 
To the same extent, we recommend to use them for the analysis of \cp data.
To this end, the R Package \textit{SimCredRegR} provides a fast implementation of \greedy and plotting routines that can be applied to \cp samples without further ado.

\begin{center}
{\large\bf SUPPLEMENTARY MATERIAL}
\end{center}

The R Package \textit{SimCredRegR}, which comes with the supplementary material, provides all the datasets and sampling algorithms that were used throughout this paper.
It further enables the computation of credible regions according to \greedy, joined highest density regions, marginal inclusion probabilities, Bonferroni sets  and stepR's joined confidence intervals.
Besides you can find the R Package \textit{SimCredRegILPR} to compute credible regions according to the ILP.
Several video (``.mp4'') files which demonstrate how credible regions evolve at different parameter choices can be found.
We further provide a proof of Theorem \ref{theo:np} and the correctness of the ILP.
We examine two more examples:  well-log data and coal mining disasters data. 
By means of different realisations of the data in Figure \ref{fig:data}, we also provide a comparison of our credible regions and \stepR's confidence intervals.
This can be found in the file ``collection\_of\_different\_data\_simulations.pdf''.
The supplement is available at:  https://github.com/siemst/simcredreg .

\bibliographystyle{chicago}      
\bibliography{biblio}   


\end{document}